\documentclass{amsart}
\usepackage{amssymb}

\newcommand{\Ps}{\mathbf{P}}

\newcommand{\C}{\mathbf{C}}

\newcommand{\cL}{\mathcal{L}}

\newtheorem{theorem}{Theorem}[section]

\newtheorem{proposition}[theorem]{Proposition}

\theoremstyle{remark}
\newtheorem{example}[theorem]{Example}
\newtheorem{remark}[theorem]{Remark}

\DeclareMathOperator{\Aut}{Aut}

\DeclareMathOperator{\rank}{rank}
\DeclareMathOperator{\CH}{CH}
\DeclareMathOperator{\sing}{sing}
\DeclareMathOperator{\length}{length}

\title{Nodal surfaces with obstructed deformations}
\author{Remke Kloosterman}
\address{Universit\`a degli Studi di Padova,
Dipartimento di Matematica,
Via Trieste 63,
35121 Padova, Italy}
\thanks{The author would like to thank Arnaud Beauville for pointing out the paper \cite{DimNodalHS}. The author would  thank the referee for many valuable suggestions to improve the exposition.}
\begin{document}
\begin{abstract}
In this text we show that the deformation space of a nodal surface $X$ of degree $d$ is smooth and of the expected dimension if $d\leq 7$ or $d\geq 8$ and $X$ has at most $4d-5$ nodes. (The case $d\leq 7$ was previously covered by Alexandru Dimca using different techniques.)

For $d\geq 8$ we give explicit examples of nodal surfaces with $4d-4$ nodes, for which the tangent space to the deformation space has larger dimension than expected.
\end{abstract}
\maketitle

\section{Introduction}
In his recent PhD thesis Yan Zhao \cite{ZhaoThesis} studied the deformation theory of nodal surfaces. He showed that for particular families of sextic surfaces the dimension of the deformation space equals the expected dimension. 
There are various deformation spaces for singular varieties (i.e., either preserving the analytic or the topological types of the singularities). However, for nodal varieties these spaces coincide and we call them ``the deformation space".

Zhao's results  motivated us to prove the  following result.

\begin{theorem} Let $X\subset \Ps^3$ be a nodal surface of degree $d$. If $d\leq 7$ or $d\geq 8$ and $X$ has at most $4(d-1)-1$ nodes then the deformation space of $X$ is smooth and  has the expected dimension.
\end{theorem}

The case $d\leq 7$ was proven by Dimca \cite{DimNodalHS} using different techniques.
The above bound is sharp:
\begin{theorem} Suppose $X\subset \Ps^3$ is a nodal surface of degree $d\geq 8$ with $4(d-1)$ nodes. Then the tangent space to the deformation space does not have the expected dimension if and only if the locus of the nodes form a complete intersection of multidegree $(1,4,d-1)$.
\end{theorem}

The proof relies on the fact that the difference between the expected dimension and the actual dimension of the tangent space at a surface $X$ equals the defect of the linear system of degree $d$ polynomials through the nodes of $X$. Then we use some ideas from \cite[Proof of Theorem 4.1]{KloNod} to bound the Hilbert function of an ideal $I$, which has defect in certain degree, and has finite base locus in a lower degree.

As an application we discuss surfaces of the form
\[ f_1f_2+f_3^2f_4\]
with $\deg(f_1)=1$ and $4\leq \deg(f_3) \leq d/2$. (See Example~\ref{exa}.)
We show that in the extreme case $\deg(f_3)=4$ these surfaces form the singular locus of the deformation space, whereas in the other extreme case $\deg(f_3)=d/2$ we show that the deformation space is smooth but not of the expected dimension.

Deformation theory of singular hypersurfaces have been studied extensively. The starting point for deformation of nodal surfaces is the paper \cite{BurnsWahl}. For our applications, the slightly different presentation in \cite[Chapter 3]{ZhaoThesis} is more suitable.

\section{Preliminaries}
Let $S=\C[x_0,\dots,x_n]$. For a homogeneous ideal $I\subset S$ denote with $h_I(k)=\dim S/I_k$, the Hilbert function of $I$. 
Let 
\[ 0 \to F_s \to F_{s-1}\to \dots \to F_0  \to S/I \to 0 \]
be a free resolution of $S/I$, where $F_i =\oplus S(-j)^{\beta_{i,j}}$. If the resolution is minimal then the $\beta_{i,j}$ are called the Betti numbers of $S/I$.
Under the weaker assumption that each $F_i$ is finitely generated we set $B_j =\sum_{i=0}^s (-1)^i \beta_{i,j}$. Then 
\[ h_I(k)=\sum_{j=0}^k B_j \binom{n+k-j}{n}\]
This formula allows us to recover $B_k$ from $h_I(k)$ and the $B_j$ with $j<k$. In particular, the $B_j$ are independent of the resolution.
The Hilbert polynomial $p_I(x)$ of $I$ equals
\[ \sum_{j\geq 0} B_j  \frac{(n+x-j)(n+x-j+1)\dots (x-j+1)}{n!}.\]
The \emph{defect} of $I$ in degree $k$ is the difference $p_I(k)-h_I(k)$ and is denoted by $\delta_k(I)$. This equals
\[ \sum_{j\geq k+n+1} B_j  \frac{(n+k-j)(n+k-j+1)\dots (k-j+1)}{n!}\]
which in turn equals
\[
(-1)^n \sum_{j\geq k+n+1} B_j \binom{j-k-1}{n}.\]
In particular, if $I$ has defect in degree $k$ then for some $j\geq k+n+1$ we have $B_j\neq 0$.

\begin{proposition}\label{prpBnd} Let $\Sigma \subset \Ps^n$ be a closed subscheme of dimension 0. Let $I$ be the ideal of $\Sigma$. Suppose that $h_I(k-n)\neq \length(\Sigma)$. Let $t$ be a positive integer and let $f_0,\dots,f_n \in S_t$ be  polynomials without a common zero. Let $\varphi: \Ps^n \to \Ps^n$ given by $(f_0,\dots,f_n)$ be the associated morphism. Let $I_t$ be the ideal of $\varphi^{-1}(\Sigma)$. Then
\[ \length(\varphi^{-1}(\Sigma))-h_{I_t}(tk-n-1)\geq \binom{n+t}{n}-n.\]
\end{proposition}
\begin{proof}
Let $X$ be a scheme and let $\ell$ be a linear form  such that $X \not \subset V(\ell)$. Then multiplication by $\ell$ yields an injective linear map $S/I(X)_k \to S/I(X)_{k+1}$. In particular, the Hilbert function of $I(X)$ is increasing. If $\dim X=0$ then the Hilbert polynomial is constant and therefore the defect is a decreasing function.

Therefore it suffices to prove the result for the highest possible value for $k$, i.e., we may assume that $h_I(j)=\length(\Sigma)$ for $j\geq k-n+1$. Let $B_i(I)$ be as defined above. Then $B_i(I)=0$ for $i\geq k+2$. By assumption, $\delta_{k-n}(I)$ and $\delta_{k-n-1}(I)$ are positive. Hence
\[ (-1)^n B_{k+1}(I) \mbox{ and }  (-1)^n(B_{k+1}(I) (n+1)+B_k(I))\]
are positive. 
The pull back of a free resolution of $S/I$ is a free resolution of $S/I_t$. In particular, $B_j(I_t)=0$ if $t\nmid j$ and $B_{tj}(I_t)=B_j(I)$. From this it follows that  $\delta_{tk-n}(I_t)$ equals
\[(-1)^ n\left(B_{t(k+1)}(I_t)\binom{n+t}{n}+B_{tk}(I_t)\right)=(-1)^n\left(B_{k+1}(I) \binom{n+t}{n}+B_k(I)\right).\] This is at least
\[ (-1)^n B_{k+1}(I) \left( \binom{n+t}{n}-n-1\right)+1\geq \binom{n+t}{n}-n.\]
\end{proof}

Macaulay \cite{Mac} described the possible Hilbert functions of homogeneous ideals. Gotzmann \cite{Gotz} described what happens in the extreme case. As a corollary we obtain the following result. (For the deduction see \cite[Section 2]{KloNod}.)
\begin{theorem}[Macaulay-Gotzmann] \label{thmHilb}  Let $I\subset S$ be an ideal. If $h_I(k)\leq k$ then $h_I(k+1)\leq h_I(k)$. If, moreover, $I_{k+1}$ is base point free then $h_I(k+1)<h_I(k)$ or $h_I(k)=0$.
\end{theorem}
\begin{remark}
This result is a key ingredient in various proofs of the explicit Noether-Lefschetz Theorem (e.g. \cite{GreenF})  and is also used in the author's proof \cite{KloNod} for the fact that a nodal hypersurfaces of degree $d$ in $\Ps^4$ with less than $(d-1)^2$ nodes is factorial.
\end{remark}

In the following we use the Alexander polynomial of hypersurfaces with isolated singularities. Suppose $X=V(f)\subset \Ps^n$ is a hypersurface with isolated singularities. Then $f$ defines an affine hypersurface in $\mathbf{A}^{n+1}$ with a one-dimensional singular locus. Let $F=V(f-1)\subset \mathbf{A}^{n+1}$ be the affine Milnor fiber. 
Fix a primitive $d$-th root of unity $\zeta$. Let $T$ be the map multiplying each $x_i$ with $\zeta$. Then $T$ acts on $F$. The Alexander polynomial $\Delta_X(t)$ of $X$ is the characteristic polynomial of $T^*$ acting on $H^{n-1}(F)$.
\begin{proposition}\label{prpGlobalDiv}  Let $X\subset \Ps^n$ is a hypersurface with isolated singularities. Then the only zeros of the Alexander polynomial are $d$-th roots of unity. 
Let $\alpha$ be  a root of unity different from one. Then one can bound the exponent of $(t-\alpha)$ in $\Delta_X$ by
\[  (-1)^{n+1}\frac{1-(-1)^n(d-1)^{n}}{d}.\]
\end{proposition}
\begin{proof}
 Let $H$ be a general hyperplane then it follows from \cite[Theorem 4.1.24]{Dim} that $\Delta_X(t)$ divides the characteristic polynomial of the monodromy of the Milnor fiber $F'$ associated with $X\cap H$.
From \cite[Example 4.1.23]{Dim} it follows that the latter polynomial equals 
$(t-1)^{(-1)^{n}}(t^d-1)^{\chi}$, with $\chi=(-1)^{n+1}\frac{1-(-1)^n(d-1)^{n}}{d}$.
 \end{proof}

In the surface case ($n=3$) we find that the exponent is at most $d^2-3d+3$. The following result explains how to calculate the Alexander polynomial of a nodal hypersurface:

\begin{theorem}[{Dimca, \cite[Theorem 6.4.5]{Dim}}]\label{thmAlex}
Let $X\subset \Ps^n$ be a nodal hypersurface of degree $d$. Let $\Sigma \subset X$ be the locus of the nodes. Then the Alexander polynomial $\Delta_X(t)$ of $X$ equals 
\[ (t-(-1)^{n})^\delta\]
with $\delta=0$ if $dn$ is odd and $\delta$ equals the defect of $I(\Sigma)$ in degree $\frac{nd}{2} -n-1$ if $dn$ is even.
\end{theorem}

\section{Deformations of nodal surfaces}
In \cite{BurnsWahl} and \cite{ZhaoThesis} the deformation theory of a degree $d$ nodal surface $X$ in $\Ps^3$ is studied. In both texts the authors consider deformations of the minimal resolution of singularities of $X$ (as an abstract variety) preserving the $-2$ curves. Moreover, Zhao studies deformations of $X$ as $V$-manifold. 

It turns out that for $d\geq 5$ the deformation space is unobstructed if and only if the ideal of the nodes of $X$ does not have defect in degree $d$.

A different approach to study deformation of hypersurfaces with isolated singularities is due to Greuel and several collaborators including Shustin, and Shustin's students. The first paper taking this approach seems to be \cite{GrK}. Although most of their results are for plane curves only, some of their results hold for arbitrary hypersurfaces. In this case they fix an ambient variety $Y$ and a divisor $\cL$ on $Y$. They consider two equivalence relations among singularities, namely equianalytic and equisingular. Then for a choice of one of the equivalence relations and fixed types $T_1,\dots T_s$ they consider the locus $V_{|\cL|}(T_1,\dots,T_s)$ of hypersurfaces in $|\cL|$ with singularities of type $T_1,\dots,T_s$.  

Consider now the case of  $s$ nodes. Then the two equivalence relations coincide and  we denote with $V_{|\cL|}(s A_1)$ the subspace of $|\cL|$ of hypersurfaces with $s$ nodes. We expect that this locus has codimension $s$ inside $\cL$. In e.g., \cite{GouGou} it is shown that the codimension of the tangent space at $X\in V_{|\cL|}(sA_1)$ equals $s-\delta$ where $\delta$ is the defect of the linear system of polynomials in $|\cL|$ vanishing at the $s$ nodes of $X$. I.e., the difference between expected and actual dimension of the tangent space is the same in both theories.

The main differences between the two approaches is the following: in the first approach one finds that the tangent space to the deformation space can be identified with $(I/J)_d$, where $I$ is the ideal of the nodes and $J$ is the Jacobian ideal of $X$, and in the second approach the tangent space can be identified with $I_d/\C F$, where $F$ is the defining polynomial of $X$. 
If $d\geq 3$ then $\dim J_d=16$ (see \cite[Corollary 4.3]{DimNodalHS} or \cite[Proposition 3.3.8]{ZhaoThesis}). In this case  one has a natural identification between $J_d/F$ and the tangent space to $\Aut(\Ps^3)$. 

In the cases $d=2$ and $d=4$ one has to be a bit more careful. In the case $d=2$ a surface with a node has infinite automorphisms group and in the case $d=4$ there is a difference between embedded deformations (i.e., deformations as a surface in  $\Ps^3$) and abstract deformations. 

\section{Proof of Theorems}

The following proposition is the same as \cite[Corollary 1.3]{DimNodalHS}, but our proof is different from Dimca's proof.

\begin{proposition}Let $X$ be a nodal surface of degree $d$. Let $\Sigma$ be the locus of points where $X$ has a node and let $I=I(\Sigma)$. Then $h_I(k)=\#\Sigma$ for $k> \frac{3}{2}d-4$. 

In particular, if $h_I(d)<\#\Sigma$ then $d\geq 8$.
\end{proposition}
\begin{proof} 
Let $s$ be an even integer. Take a general degree-$s$ base change $\varphi:\Ps^3\to \Ps^3$. Let $\tilde{I}$ be the ideal of the nodes of $\tilde{X}:=\varphi^{-1}(X)$. Since $\varphi$ is general we may assume that $\tilde{I}$ is the pull back of $I$ under $\varphi$.

If $h_I(k)<\# \Sigma$ for some $k> \frac{3}{2}d-4$ then $\delta_{\tilde{I}}(\frac{3}{2}sd-4)\geq \binom{s+3}{3}-3$ by Proposition~\ref{prpBnd}.

By Theorem~\ref{thmAlex} we have that $\delta_{\tilde{I}}(\frac{3}{2}sd-4)$ is the exponent of $(t+1)$ in the Alexander polynomial of $\varphi^{-1}(X)$ and by Proposition~\ref{prpGlobalDiv} this exponent is at most $(sd)^2-3sd+3$. For $s$ sufficiently large this is smaller than $\binom{s+3}{3}-3$ and we have a contradiction. Hence $h_I(k)=\#\Sigma$ for $k>\frac{3}{2}d-4$.
%
%
%
\end{proof}

\begin{remark} This proposition is the same as \cite[Corollary 1.3]{DimNodalHS}. Dimca bounds the degree of the syzygies by using spectral sequences and an argument comparing the Hodge filtration and the pole order filtration on the cohomology of the complement of the affine Milnor fiber associated with the defining polynomial of $X$.
\end{remark}

\begin{proposition} Let $X$ be a nodal surface of degree $d$. Suppose $d\geq 8$. Let $I$ be the ideal of the nodes of $X$. If $X$ has at most $4d-5$ nodes then $h_I(d)=\#\Sigma$. 
If $X$ has $4d-4$ nodes and $h_I(d)<\#\Sigma$ then $I$ is a complete intersection ideal of degree $(1,4,d-1)$.
\end{proposition}

\begin{proof}
Let $\ell$ be a general linear form. Let $I_H=(I,\ell)$. Consider the exact sequence
\[ 0 \to (S/I)_k \stackrel{\ell}{\to} (S/I)_{k+1} \to (S/I_H)_{k+1} \to 0.\]
Denote with $I_H=(I,\ell)$. From this exact sequence it follows that
\begin{eqnarray}\label{eqnHilb} h_I(k)=\sum_{j=0}^k h_{I_H}(k).\end{eqnarray}

Suppose now that $h_I(d)<\#\Sigma$.
If for some $k\geq 0$ we have that $h_{I_H}(k)=0$ then for all $m\geq k$ we have $h_{I_H}(m)=0$. In particular, $h_I(m)$ is constant for $m\geq k-1$. Since $h_I(d)<\#\Sigma$ and $h_I(k)=\#\Sigma$ for $k$ large, we find that $h_{I_H}(d+1)\neq 0$.

We construct now the following ideal $J$. Take $J_{d+1}$ such that $(I_H)_{d+1}\subset J_{d+1}$ and $h_J(d+1)=1$. For $k< d+1$ set
\[ J_k=\{f\in S_k \colon f S_{d+1-k} \subset J_{d+1}\} \]
and set $J_k=S_k$ for $k>d+1$. Now set $J=\oplus_k J_k$. Then $J$ contains $I_H$. Moreover, $J$ is closed under addition and under multiplication by $-1$. To show that $J$ is an ideal we have to show that for every $f\in J_k, g\in S_m$ the product $fg$ is contained  in $ J_{k+m}$. If $k+m\geq d+2$ then $J_{k+m}=S_{k+m}$ and there is nothing to prove. If $k+m\leq d+1$ then  for every $h\in S_{d+1-k-m}$ we have $(fg)h=f(gh)\in fS_{d+1-k} \subset J_{d+1}$. Hence $fg\in J$ and  $J$ is an ideal.

We constructed $J$ in such a way that for any integer $k$ the pairing $(S/J)_k \times (S/J)_{d+1-k} \to (S/J)_{d+1}\cong \C$ is perfect, i.e., $S/J$ is Artinian Gorenstein with socle degree $d-1$. In particular,  $h_J(k)=h_J(d+1-k)$.

The vector space $I_{d-1}$ contains all partial derivatives of the defining polynomial of $X$. In particular, it has a finite base locus and therefore $(I_H)_{d-1}$ has empty base locus. Hence also $J_{d-1}$ has empty base locus. Since $d-2>4$ we can apply Theorem~\ref{thmHilb} to obtain $h_J(d)\geq 2, h_J(d-1)\geq 3$ and $h_J(d-2)\geq 4$. 
Applying Theorem~\ref{thmHilb} again we find that $h_J(k)\geq 4$ for $k\geq 4$. Using that $S/J$ is Aritinian Gorenstein we obtain that $h_J(3)=h_J(d-2)\geq 4$. Similarly we find $h_J(k)\geq k+1$ for $k=0,1,2$.
 All in all we find
\[ \#\Sigma\geq h_I(d+1)=\sum_{k=0}^{d+1} h_{I_H}(k)\geq \sum_{k=0}^{d+1} h_J(k) \geq 4d-4.\]
Hence if $\#\Sigma<4d-4$ then $h_I(d)=\#\Sigma$.

Consider now the extreme case, i.e.,  $h_I(d)<\#\Sigma$ and $\#\Sigma=4(d-1)$. Then from the above discussion it follows that $I_H=J$ and we have that 
\[ h_{I_H}(k) = \left\{ \begin{array}{cl}
k+1 &\mbox{for } 0\leq k \leq 2 \\
4 & \mbox{for } 3\leq k \leq d-2  \\
d+2-k & \mbox{for } d-1\leq k \leq d+1\\ \end{array} \right.\]
Using (\ref{eqnHilb}) we find that $h_I(k)=\frac{1}{2}(k+1)(k+2)$ for $k\in \{0,1,2\}$, that $h_I(k)=4k-2$ for $3\leq k \leq d-2$ and $h_I(d-1)=4(d-1)-2-1$.
Hence $h_S(1)-h_I(1)=1$ and therefore $I$ has a generator $f_1$ in degree $1$. The ideal generated by $f_1$ and the ideal $I$ have the same Hilbert function up to degree 3, but their Hilbert functions differ in degree $4$, so $I$ has another generator $f_2$ in degree $4$. The ideal $(f_1,f_2)$ is the ideal of a quartic plane curve. Its Hilbert function equals $4k-2$ for $k\geq 4$. Hence $I_k=(f_1,f_2)_k$ for $k\leq d-2$. However, for degree $k-1$ we find that $h_I(k-1)$ is one less than   the Hilbert function of $(f_1,f_2)$. Hence there is a third generator for $I$ in degree $d-1$. As remarked above we have that $I_{d-1}$ has finite base locus. Hence $f_1,f_2,f_3$ define a complete intersection, containing $\Sigma$. The length of this complete intersection equals $4(d-1)$. Since $\Sigma$ has the same length we find that $\Sigma$ is a complete intersection of multidegree $(1,4,d-1)$.
\end{proof}
\begin{remark}
In the case that $h_I(d)<\#\Sigma$ and $\#\Sigma=4(d-1)$ we have that $I=(g_1,g_2,g_3)$, with $\deg(g_1)=1,\deg(g_2)=4,\deg(g_3)=d-1$. Since $I$ is a complete intersection ideal we have that $I^2$ is generated by products of the generators of $I$. From $f\in I^2$ and $\deg(g_3^2)>\deg(f)$ it now follows that $f\in(g_1,g_2)$, i.e., the quartic curve containing all the nodes is contained in $X$.
\end{remark}

\begin{proof}[Proof of the Theorems]

Let $X\subset \Ps^3$ be a nodal surface.
As argued in the previous section we have 
 that the tangent space to the deformation space of $X$ has the expected dimension if and only if the ideal of the nodes of $X$ does not have defect in degree $d$.
 By the previous two propositions the latter can only happen if $d\geq 8$ and $\#\Sigma\geq 4d-4$.
If $d\geq 8$ and $\#\Sigma=4d-4$ holds and the tangent space has too large dimension if and only if the ideal of the nodes is a complete intersection of degree $(1,4,d-1)$. 
\end{proof}

Finally, we would like to discuss whether the defect of $I$ in degree $d$ can be explained by singularities of the deformation space or is induced by deformation spaces with excessive dimension.

\begin{example}\label{exa}
Let $d\geq 8$ and let $4\leq a\leq \frac{1}{2}d$.
Suppose now we have a polynomial $f$ of degree $d$, such that the singular locus of $V(f)$ contains a complete intersection $\Sigma$ of multidegree $(1,a,d-1)$.
For the moment assume that the plane containing $\Sigma$ is given by $x_0=0$. Then we can find $f_2\in \C[x_1,x_2,x_3]_a,g\in \C[x_1,x_2,x_3]_{d-1}$  such that $I(\Sigma)=(x_0,f_2,g)$. 

Since $f$ has double points at $\Sigma$ we find that $f\in (x_0,f_2,g)^2$. From the fact that $(x_0,f_2,g)$ is a complete intersection ideal it follows that its square is generated by all products of two generators of $(x_0,f_2,g)$.
Each generator of degree at most $d$ is divisible by $x_0$ except for $f_2^2$. In particular, there exist polynomials $f_1\in \C[x_0,x_1,x_2,x_3]_{d-1}$ and $f_3\in \C[x_1,x_2,x_3]_{d-2a}$ such that
\[ f=x_0f_1+f_2^2 f_3.\]
This decomposition is unique, except for the fact that we may multiply $f_3$ by a non-zero constant and divide $f_2$ by the square of the same constant. Conversely, every $f$ of the above form defines a surface singular at $x_0=f_1=f_2=0$.
Hence the locus $L_0\subset S_d$ of polynomials $f$ such that $V(f)_{\sing}$ contains a complete intersection $\Sigma$ of multidegree $(1,a,d-1)$ and $\Sigma\subset V(x_0)$ has dimension
\[ \binom{d-1+3}{3} +\binom{a+2}{2}+\binom{d-2a+2}{2}-1.\]
If we drop the condition $\Sigma\subset V(x_0)$ we find a locus $L$  of dimension  
\[ \binom{d-1+3}{3} +\binom{a+2}{2}+\binom{d-2a+2}{2}+2.\]
Its codimension in $S_d$ equals
\[\frac{1}{2} (-5a^2+3a+4da -6) .\]
For small $a$ this codimension is larger than the number of nodes. Hence  for  small $a$ 
we have that $L$ is a proper closed subset of some component of the deformation space. The smallest $a$ possible is $a=4$. In that case we have by the results of this section that $h_I(d)=4(d-1)$. In particular, for $a=4$ we have  $L$ is precisely the singular locus of the space of degree $d$ surfaces with $4(d-1)$ nodes.

For 
\[  \frac{1}{10} 2d+5+\sqrt{4d^2+20d-95}< a < \frac{1}{2}d\]
we have that the codimension of $L$ in $S_d$ is less than $a(d-1)$.  However, the codimension of $L$ in $S_d$ is strictly larger than $h_I(d)$. So either $L$ is a component of the deformation space, and this component is nonreduced or $L$ is a proper closed subset in a larger component.

Finally, if $d$ is even and $a=\frac{1}{2}d$ then the codimension of $L$ in $S_d$  equals $h_I(d)$. In particular, $L$ is an irreducible component of the deformation space and this component is smooth and \emph{not} of the expected dimension.

We have two geometric-topological characterisations for the case $a=\frac{1}{2}d$. 

If $a<1/2d$ then we have that $\rank \CH^1(X)>1$, (i.e., the hyperplane section and the degree $a$ curve $f_1=f_3=0$ are linearly independent in $\CH^1(X)$), whereas for $a=\frac{1}{2}d$ we have $\rank \CH^1(X)=1$ for very general $X$. 

A second characterisation is given by the Alexander polynomial, this polynomial is a constant polynomial in the case $a<\frac{1}{2}d$ and equals $(t+1)$ in the case $a=\frac{1}{2}d$.
\end{example}

\bibliographystyle{plain}
\bibliography{remke2}

\end{document}